\documentclass[12pt]{amsart}
\usepackage[matrix,arrow]{xy}
\usepackage{amssymb}
\usepackage{amsmath}
\usepackage{amsfonts}
\usepackage{epsfig}
\usepackage{color}
\usepackage{epstopdf}
\usepackage{graphicx}
\usepackage{amsthm}
\usepackage{enumerate}
\usepackage[mathscr]{eucal}
\usepackage{verbatim}
\usepackage{bm}

\usepackage[bookmarks=true,hyperindex,pdftex,colorlinks,citecolor=blue,
linkcolor=blue,urlcolor=blue]{hyperref}
\usepackage{tikz}
\usetikzlibrary{matrix,arrows.meta}

\theoremstyle{plain}
\newtheorem{theorem}{Theorem}[section]
\newtheorem{cor}[theorem]{Corollary}
\newtheorem{prop}[theorem]{Proposition}
\newtheorem{lemma}[theorem]{Lemma}

\theoremstyle{definition}

\newtheorem{example}[theorem]{Example}

\newtheorem{question}[theorem]{Question}
\newtheorem{rem}[theorem]{Remark}
\newtheorem{remark}[theorem]{Remark}
\newtheorem{definition}[theorem]{Definition}

\newcommand{\R}{\mathbb{R}}
\newcommand{\N}{\mathbb{N}}

\newcommand{\Lin}{\mathcal{L}}

\newcommand{\eps}{\varepsilon}
\newcommand{\lam}{\lambda}

\DeclareMathOperator{\sgn}{sgn}

\DeclareMathOperator{\spann}{span}

\DeclareMathOperator{\NA}{NA}

\DeclareMathOperator{\BS}{B\check{S}}
\DeclareMathOperator{\BSa}{B\check{S}_a}

\renewcommand{\subset}{\subseteq}

\title[Norm attaining operators and the strong orthogonality]
{On a set of norm attaining operators and the strong Birkhoff-James orthogonality}

\author[G.~Choi]{Geunsu Choi}
\address[G.~Choi]{Department of Mathematics Education, Dongguk University, Seoul 04620, Republic of Korea \newline
\href{http://orcid.org/0000-0002-4321-1524}{ORCID: \texttt{0000-0002-4321-1524}}}
\email{\texttt{chlrmstn90@gmail.com}}

\author[M. Jung]{Mingu Jung}
\address[M. Jung]{School of Mathematics, Korea Institute for Advanced Study, 02455 Seoul, Republic of Korea \newline
\href{http://orcid.org/0000-0003-2240-2855}{ORCID: \texttt{0000-0003-2240-2855} }}
\email{\texttt{jmingoo@kias.re.kr}}

\author[S.~K.~Kim]{Sun Kwang Kim}
\address[S.~K.~Kim]{Department of Mathematics, Chungbuk National University, Cheongju, Chungbuk 28644, Republic of Korea\newline
	\href{http://orcid.org/0000-0002-9402-2002}{ORCID: \texttt{0000-0002-9402-2002}  }}
\email{\texttt{skk@chungbuk.ac.kr}}

\thanks{}


\begin{document}

\begin{abstract}
Continuing the study of recent results on the Birkhoff-James orthogonality and the norm attainment of operators, we introduce a property namely the adjusted Bhatia-\v{S}emrl property for operators which is weaker than the Bhatia-\v{S}emrl property. The set of operators with the adjusted Bhatia-\v{S}emrl property is contained in the set of norm attaining ones as it was in the case of the Bhatia-\v{S}emrl property. It is known that the set of operators with the Bhatia-\v{S}emrl property is norm-dense if the domain space $X$ of the operators has the Radon-Nikod\'ym property like finite dimensional spaces, but it is not norm-dense for some classical spaces such as $c_0$, $L_1[0,1]$ and $C[0,1]$. In contrast with the Bhatia-\v{S}emrl property, we show that the set of operators with the adjusted Bhatia-\v{S}emrl property is norm-dense when the domain space is $c_0$ or $L_1[0,1]$. Moreover, we show that the set of functionals having the adjusted Bhatia-\v{S}emrl property on $C[0,1]$ is not norm-dense but such a set is weak-$*$-dense in $C(K)^*$ for any compact Hausdorff $K$.
\end{abstract}

\maketitle

\section{Introduction \& Preliminaries}

The famous Bishop-Phelps theorem states that the set of norm attaining functionals on a Banach space is norm-dense in its dual space \cite{BP}. This allowed many authors to study the set of norm attaining operators, and especially J. Lindenstrauss \cite{L} first showed that there is no vector-valued version of the Bishop-Phelps theorem. However, he also found many pairs of Banach spaces such that the denseness holds, and afterwards it was also discovered that the denseness holds for many pairs of classical spaces like $(L_p[0,1],L_q[0,1])$ and $(C[0,1],L_p[0,1])$ where $p,q\in \mathbb{N}$ (see \cite{B,FP,I,S}).

In 1999, Bhatia and \v{S}emrl \cite{BS} showed that an operator $T$ on a finite dimensional complex Hilbert space is Birkhoff-James orthogonal to another operator $S$ if and only if there exists a norm attaining point $x$ of $T$ such that $Tx$ is Birkhoff-James orthogonal to $Sx$. The Birkhoff-James orthogonality is introduced by G. Birkhoff in \cite{Bir} to consider the concept of orthogonality on linear metric spaces. In general, the characterization of Bhatia and \v{S}emrl is not true for operators between Banach spaces \cite{LS}. Nevertheless, when the domain space is finite dimensional or has some geometric property such as property quasi-$\alpha$ \cite{CK} or the Radon-Nikod\'ym property \cite{K}, the set of operators for which the characterization holds is a norm-dense subset of the norm attaining operators similarly to the case of the aforementioned norm attaining operator theory. However, on some classical spaces with no (or few) extreme points such as $c_0$, $L_1[0,1]$ and $C[0,1]$, the characterization holds for only a few operators, implying that the denseness of such a set of operators does not hold (see \cite{CK,K,KL,PSJ}). This is the main difference between this study and the classical norm attaining operator theory. The main aim of the present paper is to take into account a new class of norm attaining operators which contains (properly) the set of operators with the characterization of Bhatia and \v{S}emrl, such that the denseness of such a set holds for some classical spaces.

For more details, we restart the introduction with the notions. Throughout the paper, $X$ and $Y$ are real Banach spaces. By the sets $S_X$ and $B_X$ we mean the unit sphere and the closed unit ball of a Banach space $X$, respectively, and $X^*$ stands for the topological dual space of $X$. We denote by $\Lin(X,Y)$ the space of all bounded linear operators from $X$ into $Y$. An operator $T \in \Lin(X,Y)$ is said to \emph{attain its norm} at $x_0 \in S_X$ if $\|Tx_0\|=\|T\|=\sup \{ \|Tx\| : x\in B_X \}$. The set of norm attaining operators is denoted by $\NA(X,Y)$, and we define the set of norm attaining points of an operator $T \in \Lin(X,Y)$ by $M_T := \{x \in S_X : \|Tx\|=\|T\|\}$.

We say a vector $x \in X$ is \emph{orthogonal to $y \in X$ in the sense of Birkhoff-James} or \emph{Birkhoff-James orthogonal} to $y$ if $\|x\| \leq \|x+ \lam y \|$ for any scalar $\lam$, and it is denoted by $x \perp_B y$. Motivated by the aforementioned result of Bhatia and \v{S}emrl \cite{BS}, the authors introduced in \cite{PSJ} the \emph{Bhatia-\v{S}emrl property} (in short, $\BS$ property) for a norm attaining operator $T \in \Lin(X,Y)$ that for any $S \in \Lin(X,Y)$ with $T \perp_B S$, there exists $x_0 \in M_T$ such that $Tx_0 \perp_B Sx_0$. The set of norm attaining operators between $X$ and $Y$ with the $\BS$ property is denoted by $\BS(X,Y)$.

As it is mentioned in the beginning, the set $\NA(X,Y)$ is norm-dense in $\Lin(X,Y)$ for most of classical Banach spaces $X$ and $Y$, and by definition the set $\BS(X,Y)$ is a subset of $\NA(X,Y)$. Hence, it is quite natural to ask whether the set $\BS(X,Y)$ is norm dense for the same $X$ and $Y$. As we have mentioned, the answer is known to be affirmative under some geometric conditions on $X$ like property quasi-$\alpha$ and Radon-Nikod\'ym property. However, for classical spaces $X=L_1[0,1], C[0,1]$ or $c_0$, it fails \cite{CK,K,KL}. Motivated by these observations, we are willing to consider a suitable subset of $\NA(X,Y)$ which contains $\BS(X,Y)$ and is norm-dense in $\mathcal{L} (X,Y)$ even when $\BS(X,Y)$ is not norm-dense.

In order to define such a suitable subset, we recall the concept of the strong orthogonality of vectors which was recently  taken into consideration in \cite{PSJ,SPJ} with a notion of the Birkhoff-James orthogonality. We say that a vector $x \in X$ is \emph{strongly orthogonal to $y \in X$ in the sense of Birkhoff-James} if $\|x\| < \|x+ \lam y \|$ for any $\lam \neq 0$, and it is denoted by $x \perp_S y$. By the definition, we see that if $x \perp_S y$, then $x \perp_B y$. This concept is used to characterize the strict convexity of a Banach space which we present below.

\begin{rem}[\mbox{\cite[Theorem 2.4]{SPJ}}]\label{rem:str-cvx}
Let $X$ be a Banach space. Then, $x \perp_S y$ is equivalent to $x \perp_B y$ for every $x,y \in X \setminus \{0\}$ if and only if $X$ is strictly convex.
\end{rem}

Using the concept of strong orthogonality, we define the main object of this paper.

\begin{definition}
A bounded linear operator $T \in \Lin(X,Y)$ is said to have the \emph{adjusted Bhatia-\v{S}emrl property} \textup{(}in short, \emph{adjusted $\BS$ property}\textup{)} if it attains its norm and for any $S \in \Lin(X,Y)$ with $T \perp_S S$, there exists $x_0 \in M_T$ such that $Tx_0 \perp_B Sx_0$. The set of operators with the adjusted $\BS$ property is denoted by $T \in \BSa(X,Y)$.
\end{definition}

It is clear that
$$
\BS(X,Y) \subseteq \BSa(X,Y) \subseteq \NA(X,Y)
$$
for all Banach spaces $X$ and $Y$. We will see later that both inclusions are proper in many cases. In \cite{KL}, the authors characterized the set of operators with the $\BS$ property in terms of other known materials, and we have analogies. All of them are easy consequence of the above inclusion or the own proofs of \cite[Corollary 2.2, Corollary 2.3, Proposition 2.4, Proposition 2.5]{KL}, so we omit their proofs.

\begin{prop}
Let $X$ and $Y$ be Banach spaces.
\begin{enumerate}
\setlength\itemsep{0.3em}
\item[\textup{(a)}] $X$ is reflexive if and only if $\BSa(X,\R) =X^*$. 
\item[\textup{(b)}] $T \in \BSa(X,Y)$ if and only if $T$ attains its norm and for any $S \in \Lin(X,Y)$ with $T \perp_S S$, there exist $x_0 \in S_X$ and $y_0^* \in S_{Y^*}$ such that $y_0^*(Tx_0)=\|T\|$ and $y_0^*(Sx_0)=0$.
\item[\textup{(c)}] If $\BSa(X,Y) = \Lin(X,Y)$ for every reflexive $X$, then $Y$ is one-dimensional.
\item[\textup{(d)}] $\BSa(X,Y)=\Lin(X,Y)$ for every $Y$ if and only if $X$ is one-dimensional.
\end{enumerate}
\end{prop}

Our first main result concerns the inclusions for specific spaces, and they can be described as follows.

\begin{theorem}\label{theorem:set}
Let $X$ be a Banach space.
\begin{enumerate}
\setlength\itemsep{0.3em}
\item[\textup{(a)}] If $X=c_0$, then $\BS(X,\R) \subsetneqq \BSa(X,\R) = \NA(X,\R)$.
\item[\textup{(b)}] If $X=L_1[0,1]$ or $C[0,1]$, then $\BS(X,\R) \subsetneqq \BSa(X,\R) \subsetneqq \NA(X,\R)$.
\item[\textup{(c)}] If $X$ is non-reflexive and $X^*$ is strictly convex, then $\BS(X,\R) = \BSa(X,\R) \subsetneqq \Lin(X,\R)$.
\end{enumerate}
\end{theorem}

Furthermore, we show that even the denseness holds for the spaces $c_0$ and $L_1[0,1]$, whereas there are still many norm attaining operators without the adjusted $\BS$ property. We summarize the main denseness results of the present paper below which distinguishes the adjusted $\BS$ property from the original $\BS$ property. Here, we present only the cases that the range space is the scalar field for simplicity and we refer to the next section for more generalized results on range spaces.

\begin{theorem}\label{theorem:denseness}
Let $X$ be a Banach space.
\begin{enumerate}
\setlength\itemsep{0.3em}
\item[\textup{(a)}] If $X=c_0$, then $\BSa(X,\R)$ is norm-dense in $\Lin(X,\R)$.
\item[\textup{(b)}] If $X=L_1[0,1]$, then $\BSa(X,\R)$ is norm-dense in $\Lin(X,\R)$.
\item[\textup{(c)}] If $X=L_1[0,1]$, then the set of norm attaining operators without the adjusted $\BS$ property is norm-dense in $\Lin(X,\R)$.
\item[\textup{(d)}] If $X=C[0,1]$, then $\BSa(X,\R)$ is weak-$*$-dense in $\Lin(X,\R)$.
\item[\textup{(e)}] If $X=C[0,1]$, then the set of norm attaining operators without the adjusted $\BS$ property is norm-dense in $\Lin(X,\R)$.
\end{enumerate}
\end{theorem}

\section{Main results}

In this section, we provide results as mentioned in Theorems \ref{theorem:set} and \ref{theorem:denseness} of the previous section. We begin with an immediate result that the two sets of the $\BS$ properties are the same. We omit its proof since it follows directly from Remark \ref{rem:str-cvx} and definitions of the $\BS$ property and the adjusted $\BS$ property.
\begin{prop}\label{prop:str-cvx}
Let $X$ and $Y$ be Banach spaces. If $\Lin(X,Y)$ is strictly convex, then $\BS (X, Y) = \BSa (X, Y)$. 
\end{prop}

\begin{cor}
If $X$ is a separable Banach space, then there is an equivalent renorming $\tilde{X}$ of $X$ such that $\BS(\tilde{X},\R)=\BSa(\tilde{X},\R)$.
\end{cor}

\begin{proof}
According to \cite[II. Theorem 2.6]{DGZ}, there is an equivalent norm $| \cdot |$ on $X$ such that $(X, |\cdot|)^*$ is strictly convex. Thus the conclusion follows by Proposition \ref{prop:str-cvx}.
\end{proof}

From the well known characterization of reflexivity by R.C. James \cite{J2} that $X$ is reflexive if and only if $\NA(X,\R)=\Lin(X,\R)$, we have an example that $\BSa(X,\R) \subsetneqq \Lin(X,\R)$.
\begin{example}
For a non-reflexive Banach space $X$ such that $X^*$ is strictly convex such as a suitable renorming of $c_0$, we have
$$
\BS(X,\R) = \BSa(X,\R) \subset \NA(X,\R)\subsetneqq \Lin(X,\R).
$$
\end{example}

The converse of Proposition \ref{prop:str-cvx} is not true in the case that $X$ and $Y$ are a reflexive space and a strictly convex space respectively such that $\Lin (X, Y) = \mathcal{K} (X, Y)$ where $\mathcal{K} (X, Y)$ is the space of compact operators in $\Lin (X, Y)$. In order to see this, we need the following modification of \cite[Theorem 2.1]{GSP}. 
\begin{theorem}\label{rem:str-cvxop}
Let $X$ be a reflexive Banach space and $Y$ be a strictly convex Banach space. If $T,S\in \mathcal{K} (X, Y)$ satisfy $T\perp_B S$, then $T\perp_S S$ or $Sx=0$ for some $x\in M_T$.
\end{theorem}
In \cite[Theorem 2.1]{GSP}, the authors considered the case that $X=Y$ which is both reflexive and strictly convex, but their arguement can be applied for the setting of Theorem \ref{rem:str-cvxop}. 

\begin{theorem}\label{theorem:GSP}
Let $X$ be a reflexive Banach space and $Y$ be a strictly convex Banach space. If $\Lin (X, Y) = \mathcal{K} (X, Y)$, then $\BS (X, Y) = \BSa (X, Y)$. 
\end{theorem}

\begin{proof} 
Since we only need to prove $\BSa (X, Y)\subset \BS (X, Y)$, we fix $T\in \BSa (X, Y)$ and $S\in \Lin (X, Y) $ so that $T\perp_B S$. From  Theorem \ref{rem:str-cvxop}, we have $T\perp_S S$ or $Sx_0=0$ for some $x_0\in M_T$. This shows that  $Tx_1\perp_B Sx_1$ for some $x_1\in M_T$.
\end{proof} 

\begin{example}
Let $X$ and $Y$ be Banach spaces. We have that $\BS (X, Y) = \BSa (X, Y)$ whenever
\begin{enumerate}
\setlength\itemsep{0.3em}
\item[\textup{(a)}] $X$ is finite dimensional and $Y$ is strictly convex. 
\item[\textup{(b)}] $X$ is reflexive and $Y$ is finite dimensional strictly convex.
\item[\textup{(c)}] (Pitt) $X$ is a closed subspace of $\ell_p$ and $Y$ is a closed subspace of $\ell_r$ with $1 < r< p< \infty$.
\item[\textup{(d)}] (Roshenthal) $X$ is a closed subspace of $L_p (\mu)$ and $Y$ is a closed subspace of $L_r (\nu)$, $1 < r < p < \infty$ and 
 \begin{itemize}
 \setlength\itemsep{0.3em}
   \item[\textup{(d1)}] $\mu$ and $\nu$ are atomic --also covered by (c)--,
   \item[\textup{(d2)}] or $1\leq r<2$ and $\nu$ is atomic,
   \item[\textup{(d3)}] or $p>2$ and $\mu$ is atomic.
    \end{itemize}
\end{enumerate}
\end{example}

Item (c) can be found in \cite[Theorem~2.1.4]{Albiac-Kalton}, for instance; and the item (d) can be found in the paper by H.~Rosenthal \cite[Theorem~A2]{Rosenthal-JFA1969}. 

~

We are now interested in operators defined on classical spaces. To do so, the following result becomes a useful tool of extending the range space from $\R$ to some specific Banach spaces. It contains a concept of \emph{property quasi-$\beta$}, which was first introduced in \cite{AAP1996} as a weakening of \emph{property $\beta$} of Lindenstrauss \cite{L}. We note that $c_0, \ell_\infty$ and finite dimensional spaces with a polyhedral unit ball are examples of Banach spaces satisfying property $\beta$ (hence property quasi-$\beta$) and every closed subspace of $c_0$ has property quasi-$\beta$ (see \cite[Example 3.2]{JMR}).

\begin{prop}\label{prop:beta}
Let $X$ be a Banach space such that $\BSa(X,\R)$ is norm-dense in $\Lin(X,\R)$, and let $Y$ be a Banach space with property quasi-$\beta$. Then, $\BSa(X,Y)$ is norm-dense in $\Lin(X,Y)$.
\end{prop}

\begin{proof}
Since the proof is a minor modification of \cite[Proposition 2.8]{CK}, we just comment the key observation by adapting the same notion instead of giving full details. As in their proof, for $N=1$, we observe the assumption $B \perp_S C$ implies that $\| B^* y_{\alpha_0}^* + \lambda C^* y_{\alpha}^* \| > \| B^* y_{\alpha_0}^* \|$ for every $\lambda \neq 0$, that is, $B^* y_{\alpha_0}^* \perp_S C^* y_{\alpha_0}^*$. 
Since $B$ is constructed in such a way that $B^* y_{\alpha_0}^* \in \BSa (X, \mathbb{R})$, there exists $x \in M_{B^* y_{\alpha_0}^*}$ such that $B^* y_{\alpha_0}^* (x) \perp_B C^* y_{\alpha_0}^* (x)$. This completes the proof. 
\end{proof}

\subsection{When the domain space is $c_0$.}

In this subsection, we show that $\BSa(c_0,Y)$ can be norm-dense for some Banach space $Y$ while $\BS(c_0,Y) = \{0\}$ for many Banach spaces $Y$ such as strictly convex spaces or spaces with the Radon-Nikod\'ym property \cite[Corollary 3.6]{CK}. Moreover, it is known that $\BS(c_0,c_0)=\{0\}$ \cite[Proposition 3.7]{CK}. In contrast with this situation, we will observe that ${\BSa(c_0,c_0)}$ is norm-dense.

For more general setting, we consider $X$ as an $\ell_1$-predual space and $\phi : \ell_1 \rightarrow X^*$ be an isometric isomorphism. We denote the canonical basis of $X^*$ by $u_n^* = \phi (e_n^*)$, where $(e_n,e_n^*)$ is the canonical biorthogonal system of $c_0$. 
Note that the Banach space $c$ of convergent sequences is a typical example of an $\ell_1$-predual space while it is not isometric to $c_0$. Moreover, there exists an $\ell_1$-predual space which is not isomorphic to a $C(K)$-space \cite{BL} and an isomorphic predual of $\ell_1$ (that is, a Banach space whose dual is isomorphic to $\ell_1$) which has the Radon-Nikod\'ym property \cite{BD}.

\begin{theorem}\label{thm:ell1predual}
Let $X$ be an $\ell_1$-predual space. If $x^* \in X^*\setminus \{0\}$ is of the form $\sum_{k=1}^n a_k u_k^*$ with $n \in \N$, then $x^* \in \BSa (X, \mathbb{R})$. 
\end{theorem} 

\begin{proof}
Let $x^* =\sum_{k=1}^n a_k u_k^*$ for some $n \in \N$ with $a_k \neq 0$ for $k =1,\ldots, n$. It is not difficult to check that $x^* \in \NA (X, \mathbb{R})$ (see, for instance, \cite[Theorem 5.10]{DMRR}). Let $y^* \in X^*$ so that $x^* \perp_S y^*$. Note that $\|y^* \| = \sum_{k=1}^\infty |\phi^{-1} (y^*) (e_k) | < \infty$. 

We claim that 
\begin{equation}\label{eq:ell1}
\left| \sum_{k=1}^n \sgn (a_k) b_k \right| < \sum_{k=n+1}^\infty |b_k|,
\end{equation} 
where $b_k = \phi^{-1} (y^*) (e_k)$ for each $k \in \N$ and $\sgn$ denotes the sign of a value in $\R$. Suppose that \eqref{eq:ell1} is not true. Then, for $\lam_0$ such that $0<|\lam_0| < \min_{1 \leq k \leq n} \frac{|a_k|}{\max\{1,|b_k|\}}$ and $\sgn \lam_0 = - \sgn(\sum_{k=1}^n \sgn(a_k)b_k)$, we have that
\begin{align*}
\|x^* + \lam_0 y^*\| &= \sum_{k=1}^\infty \left| \phi^{-1} (x^*) (e_k) + \lambda_0 \phi^{-1} (y^*)(e_k) \right| \\ 
&= \sum_{k=1}^n \left| |a_k| + \lam_0 \sgn(a_k) b_k \right| + |\lam_0| \sum_{k=n+1}^\infty |b_k| \\
&= \sum_{k=1}^n |a_k| - |\lam_0| \left| \sum_{k=1}^n \sgn(a_k) b_k \right| + |\lam_0| \sum_{k=n+1}^\infty |b_k| \leq \sum_{k=1}^n |a_k| = \|x^*\|,
\end{align*}
which is a contradiction. Now, consider an element
$$
x = \left( \sgn(a_1), \ldots, \sgn(a_n), c_{n+1}, c_{n+2}, \ldots \right) \in B_{c_{00}},
$$
where $(c_k) \in B_{c_{00}}$ is chosen so that
$$
\sum_{k=n+1}^\infty c_kb_k = - \sum_{k=1}^n \sgn(a_k)b_k.
$$
This is possible since
$$
\frac{\left| \sum_{k=1}^n \sgn (a_k) b_k \right|}{\sum_{k=n+1}^\infty |b_k|} < 1.
$$
Let us say $c_k = 0$ for every $k \geq m+1$ for some $m >n$. Consider $u:= (\phi^*)^{-1} (x) \in X^{**}$. 

Set $F_m := \spann \{ u_1^*, \ldots, u_m^* \} \subseteq X^*$ and let $Q_m : X^* \rightarrow F_m$ be the canonical contractive projection. Note that $Q_m^* (F_m^*)$ is naturally identified to a subspace of $X$ as $Q_m$ is weak-$*$ continuous (see \cite[Corollary 4.1]{Gasparis}).

 We see that $v:= Q_m^* (u \vert_{F_m} ) \in X$ is an element in $M_{x^*}$ where $u \vert_{F_m}$ is the restriction of $u$ on $F_m$. Indeed, observe that 
\begin{align*}
x^* (v) = Q_m (x^*) (u) = x^* (u) = \sum_{k=1}^n a_k u_k^* (u) &= \sum_{k=1}^n a_k \phi (e_k^*) ( (\phi^*)^{-1} (x)) )\\
&= \sum_{k=1}^n a_k e_k^* (x) \\
&= \sum_{k=1}^n |a_k| = \|x^* \|. 
\end{align*} 
 On the other hand, 
\begin{align*}
y^* (v) = Q_m (y^*) (u) &= \left(\sum_{k=1}^m b_k u_k^* \right) (\phi^*)^{-1} (x) \\
&= \sum_{k=1}^n b_k \sgn (a_k) + \sum_{k=n+1}^m b_k c_k = 0. 
\end{align*} 
This shows that $x^* \in \BSa (X, \mathbb{R})$. 
\end{proof} 

For an $\ell_1$-predual space $X$, it is well known that the set of extreme points of $B_{X^*}$ is $\{ \theta u_n^* : n \in \N, \theta= \pm 1\}$. By combining Theorem \ref{thm:ell1predual} with the Krein-Milman theorem we have the following.
 
\begin{cor}\label{cor:ell1predual}
Let $X$ be an $\ell_1$-predual space. Then the set $\BSa (X, \mathbb{R})$ is weak-$*$-dense in $X^*$. 
\end{cor}

It is well known (and easy to check) that every norm attaining functional on $c_0$ is finitely supported. Hence, Theorem \ref{thm:ell1predual} shows that $\BSa(c_0,\R) = \NA(c_0,\R)$. In particular, we have the following result.

\begin{cor}\label{cor:c_0-to-R}
The set $\BSa(c_0,\R)$ is norm-dense in $\Lin(c_0,\R)$.
\end{cor}

We finish the present subsection with giving the main consequence of Proposition \ref{prop:beta} combined with Corollary \ref{cor:c_0-to-R} and raising a natural open question due to the known result on the denseness of norm attaining operators from $c_0$ to a uniformly convex Banach space \cite{Kim2013}. 

\begin{cor}
Let $Y$ be a Banach space with property quasi-$\beta$. Then, $\BSa(c_0,Y)$ is norm-dense in $\Lin(c_0,Y)$. In particular, $\BSa(c_0,c_0)$ is norm-dense in $\Lin(c_0,c_0)$.
\end{cor}

\begin{question}
For a uniformly convex Banach space $Y$, is $\BSa(c_0,Y)$ norm-dense in $\Lin(c_0,Y)$?
\end{question}

\subsection{When the domain space is $L_1[0,1]$.}

Similarly to the case of $c_0$ we show that ${\BSa(L_1[0,1],\R)}$ is norm-dense in $\Lin(L_1[0,1],\R)$ while it is shown in \cite[Proposition 3.6]{KL} that $\BS(L_1[0,1],\R) = \{0\}$.

\begin{theorem}
The set $\BSa(L_1[0,1],\R)$ is norm-dense in $\Lin(L_1[0,1],\R)$.
\end{theorem}

\begin{proof}
We claim that for a fixed $f\in L_\infty[0,1]\setminus\{0\}$ and $\varepsilon>0$ there exists $g\in\BSa(L_1[0,1],\R)$ such that $\|g-f\|<\varepsilon$. For a suitable number $0<\delta<\max\{\varepsilon,\|f\|\}$, we assume that $\sigma(\{s\in[0,1]: f(s)>\|f\|-\delta\})>0$ where $\sigma$ is the Lebesgue measure on $[0,1]$. Otherwise we apply the following proof for $-f$.

Consider a function $g \in L_\infty[0,1]$ defined by
\begin{displaymath}
g(s):=\left\{\begin{array}{@{}cl}
\displaystyle \, \|f\| & \text{if } \|f\|-\delta< f(s) \\
\displaystyle \, f(s) & \text{if } -\|f\| +\delta \leq f(s) \leq \|f\| - \delta\\
\displaystyle \, -\|f\| + \delta \ & \text{if } f(s) <-\|f\| +\delta.
\end{array} \right.
\end{displaymath}
From the construction, we have that 

\begin{enumerate}
\setlength\itemsep{0.3em}
\item[\textup{(i)}] $\|f-g\|<\delta$
\item[\textup{(ii)}] $\sigma(\{s \in [0,1] \colon g(s)=\|g\|\})>0$ and
\item[\textup{(iii)}] $|g(t)| \leq \|g\|-\delta$ or $g(t) = \|g\|$ for each $t\in [0,1]$. 
\end{enumerate}

To see $g\in\BSa(L_1[0,1],\R)$, take $h \in L_\infty[0,1]$ such that $g \perp_S h$.
For $\Phi := \{ s \in [0,1] \colon g(s)=\|g\| \}$, $\Phi^+ :=\{t \in \Phi \colon h(x)>0\}$ and $\Phi^- :=\{t \in \Phi \colon h(x)<0\}$, we see that both $\sigma(\Phi^+)$ and $\sigma(\Phi^-)$ are positive. Indeed, if $\sigma(\Phi^+) =0$, then it follows that $\|g+\lam h\|\leq \|g\|$ for $0<\lam<\delta/\|h\|$. Similarly if $\sigma(\Phi^-) =0$, then $\|g-\lam h\|\leq \|g\|$ for $0<\lam<\delta/\|h\|$.

We now define a function $\varphi \in S_{L_1[0,1]}$ by
\begin{displaymath}
\varphi(t):=\left\{\begin{array}{@{}cl}
\displaystyle \, \frac{H^-}{H^+\sigma(\Phi^-)+H^-\sigma(\Phi^+)} & \text{on } t \in \Phi^+ \\
\displaystyle \, \frac{H^+}{H^+\sigma(\Phi^-)+H^-\sigma(\Phi^+)} & \text{on } t \in \Phi^- \\
\displaystyle \, \phantom{\Big[}0\phantom{\Big]} & \text{otherwise},
\end{array} \right.
\end{displaymath}
where $H^+$ and $H^-$ are given by $H^+ = \int_{\Phi^+} h \,d\sigma >0$ and $H^- = -\int_{\Phi^-} h \,d\sigma >0$. We can easily see that $\int_0^1 g \varphi \,d\sigma =\|g\|$ and $\int_0^1 h \varphi \,d\sigma =0$, which shows that $g$ attains its norm at $\varphi$ and $\int_0^1 g \varphi \,d\sigma \perp_B \int_0^1 h \varphi \,d\sigma$. Hence, $g$ is an element in $\BSa(L_1[0,1],\R)$.
\end{proof}

\begin{cor}
For a Banach space $Y$ with property quasi-$\beta$, the set $\BSa(L_1[0,1],Y)$ is norm-dense in $\Lin(L_1[0,1],Y)$.
\end{cor}

The next result reveals not only that the set of norm attaining operators without the adjusted $\BS$ property may be nonempty but also its portion can be large. This generalizes the previously known fact \cite[Theorem 3.2]{K} that if $Y$ has the Radon-Nikod\'ym property, then the set of norm attaining operators without the B\v{S} property is norm-dense in $\mathcal{L}(L_1([0,1]),Y)$.

\begin{theorem}\label{theorem:L_1-RNP}
Let $Y$ be a Banach space with the Radon-Nikod\'ym property. Then, the set of norm attaining operators without the adjusted $\BS$ property is norm-dense in $\Lin(L_1[0,1],Y)$.
\end{theorem}

\begin{proof}

From \cite[Theorem 5, p. 63]{DU}, we consider an operator $T\in \Lin(L_1[0,1],Y)$ as a function $f\in L_\infty([0,1],Y)$ isometrically by the identification $Th=\int_0^1 fh~d\sigma$ for all $h \in L_1[0,1]$ where $\sigma$ is the Lebesgue measure.

 For a non-zero $f \in L_\infty([0,1],Y)$ and given $0<\eps<\|f\|$, it is enough to show that there exists a norm attaining function $g\in L_\infty([0,1],Y)$ without the adjusted $\BS$ property so that $\|g-f\|<\eps$.

 From \cite[Corollary 3, p. 42]{DU}, a function $f$ can be approximated by countably valued functions. Hence, there exists $f_0 = \sum_{i=1}^\infty y_i \chi_{E_i} \in L_\infty([0,1],Y)$ satisfying $\|f_0-f\|<{\eps}/{4}$ and $\|f_0\|=\|f\|$ where $E_i\subset [0,1]$ is a sequence of mutually disjoint Lebesgue measurable subsets for $i \in \N$.

Fix $k\in \mathbb{N}$ such that $\|y_k\|>\|f\|-{\eps}/{4}$ and $\sigma(E_k)>0$, take a Lebesgue measurable subset $E_0\subset E_k$ so that $0<\sigma(E_0)<\sigma(E_k)$. 
Define 
\[
f_1= z_0 \chi_{E_0} + \sum_{i \in \mathbb{N} \setminus \{k\}} z_i \chi_{E_i} + z_k \chi_{E_k \setminus E_0} \in L_\infty([0,1],Y),
\]
where
$$
z_0=\left(\|f\| -\dfrac{\eps}{4}\right) \dfrac{y_k}{\|y_k\|}, \quad z_k=\|f\|\dfrac{y_k}{\|y_k\|} \quad \text{and} \quad z_i=\left(1 -\dfrac{\eps}{4\|f\|}\right)y_i
$$
for each $i \in \N \setminus \{k\}$. From the construction, we have that
\begin{enumerate}
\setlength\itemsep{0.3em}
\item[\textup{(i)}] $\|f-f_1\|<\dfrac{\eps}{2}$,
\item[\textup{(ii)}] $\|z_k\|=\|f\|=\|f_1\|$,
\item[\textup{(iii)}] $z_0 = \left(1 -\dfrac{\eps}{4\|f\|}\right) z_k$ and
\item[\textup{(iv)}] $\|z_i\|\leq\|f\|-\dfrac{\eps}{4}$ for all $i \in \N \setminus \{k\}$.
\end{enumerate}
Put $\eta := \min\left\{ \dfrac{\eps}{4\|f\|}, \dfrac{\sigma(E_0)}{2}\right\}>0$ and fix any $t_0 \in [0,1 - \eta] \cap E_0$ such that $\sigma([t_0,t_0+\delta] \cap E_0)>0$ for any $\delta>0$. Define a new function $g \in L_\infty([0,1],Y)$ by
\begin{displaymath}
g(t):=\left\{\begin{array}{@{}cl}
\displaystyle \, \bigl[(1+t_0)-t\bigr]z_k & \displaystyle \text{if } t \in [t_0,t_0+\eta] \cap E_0, \\
\displaystyle \, f_1(t) & \text{otherwise}.
\end{array} \right.
\end{displaymath}
It is clear that $\|g-f\| \leq \|g-f_1\| + \|f_1-f\| <\eps$ and $\|g\|=\|f\|$. To see that $g$ does not have the adjusted $\BS$ property, consider a function $h \in L_\infty([0,1],Y)$ defined by
\begin{displaymath}
h(t):=\left\{\begin{array}{@{}cl}
\displaystyle \, z_k & \text{if } t \in [t_0,t_0+\eta] \cap E_0 \cap C \\
\displaystyle \, -z_k & \text{if } t \in [t_0,t_0+\eta] \cap E_0 \cap C^c \\
\displaystyle \, g(t) & \text{otherwise},
\end{array} \right.
\end{displaymath}
where $C$ is a nowhere dense closed subset of $[t_0,t_0+\eta] \cap E_0$ satisfying that
$$
0<\sigma([t_0,t_0+\delta]\cap E_0 \cap C) < \sigma([t_0,t_0+\delta] \cap E_0)
$$
for any small $0<\delta<\eta$. For instance, we may take a fat Cantor-type set distributed on $[t_0,t_0+\eta] \cap E_0$.

To show that $g \perp_S h$, we fix any $|\lambda|>0$ and obtain
\begin{align*}
\|g+\lambda h\| &\geq \operatorname{ess\,sup} \bigl\{\|g(t) + \lambda h(t)\| : t \in [t_0,t_0+\eta] \cap E_0\bigr\} \\
&\geq \max \bigl\{ \|z_k +\lambda z_k\|, \|z_k -\lambda z_k\|\bigr\} \\
&= (1+ |\lambda|) \|z_k\| > \|g\|
\end{align*}
from the construction of $h$. However, $\int_0^1 g\varphi \,d\sigma \not\perp_B \int_0^1 h\varphi \,d\sigma$ for all $\varphi \in M_g$. Indeed, as we know that
\begin{align*}
\|z_k\|=\|f\|=\|g\|&=\left\|\int_0^1 g\varphi \,d\sigma\right\|\\
&\leq \|z_0\|\int_{E_0} |\varphi| \,d\sigma+\sum_{i \in \mathbb{N} \setminus \{k\}} \|z_i\|\int_{E_i} |\varphi| \,d\sigma+\|z_k\|\int_{E_k\setminus E_0} |\varphi| \,d\sigma\\
&\leq \|z_k\|\|\varphi\|\\
&=\|z_k\|
\end{align*}
the support of $\varphi$ belongs $E_k \setminus E_0$ almost everywhere. Moreover, the fact that $h(t)=g(t)$ almost everywhere on $E_k \setminus E_0$ leads to $\int_0^1 h\varphi \,d\sigma = \int_0^1 g\varphi \,d\sigma$.
\end{proof}

We close the present subsection by raising the following question on denseness of $\BSa(L_1[0,1],Y)$. It is worth mentioning that $\NA (L_1 [0,1], Y)$ is norm-dense in $\Lin(L_1[0,1],Y)$ when a Banach space $Y$ has the Radon-Nikod\'ym property \cite{uhl}. 

\begin{question}
Is $\BSa(L_1[0,1],Y)$ norm-dense in $\Lin(L_1[0,1],Y)$ for a Banach space $Y$ with the Radon-Nikod\'ym property?
\end{question}

\subsection{When the domain space is $C[0,1]$.}

Recall from \cite[Proposition 3.7]{KL} that
$$
\BS(C[0,1],Y) \subseteq \left\{ T \in \NA(C[0,1],Y) : M_T =\left\{ \pm \chi_{[0,1]}\right\}\right\} \cup \{0\}
$$
for every reflexive strictly convex Banach space $Y$. In particular, this shows that $\BS(C[0,1],\R)$ is not norm-dense in $\Lin(C[0,1],\R)$. Similarly, we find that $\BSa(C[0,1],\R)$ is not norm-dense, but we show that the same set is weak-$*$-dense even though it is not true for $\BS(C[0,1],\R)$.

To deduce these results, we use the classical Riesz representation theorem. Indeed, we consider the space $\Lin(C[0,1],\R)$ isometrically as the space $\mathcal{M}[0,1]$ of all finite regular Borel measures on $[0,1]$ with the norm induced by the total variation $\| \mu \| = |\mu| ([0,1])$. The duality is given by
\[
\mu (f) = \int_{[0,1]} f \, d\mu.
\]
Thus, it is natural to say that a measure $\mu \in \mathcal{M}[0,1]$ is \emph{norm attaining} if there exists $f\in S_{C[0,1]}$ such that $\|\mu\| = |\mu(f)|$.

We first present some sufficient conditions for regular Borel measures on $[0,1]$ to lack the adjusted $\BS$ property.
 
\begin{prop}\label{example1}A norm attaining regular Borel measure $\mu\in \mathcal{M}[0,1]$ satisfying the following two properties does not have the adjusted $\BS$ property. 
\begin{enumerate}
\setlength\itemsep{0.3em}
\item[\textup{(i)}] there exists an interval $A=(\alpha,\beta) \subset [0,1]$ such that $F(t)=\mu\bigl((\alpha,t)\bigr)$ on $[\alpha,\beta]$ is an nondecreasing \textup{(}or nonincreasing\textup{)} continuous function which is not a constant.
\item[\textup{(ii)}] there exists an interval $C=(\gamma,\delta) \subset [0,1]\setminus A$ such that $|\mu|(C)=0.$
\end{enumerate}
\end{prop}
\begin{proof}
We only give a proof for the case that $F(t)$ is nondecreasing, since the other case can be proved by taking $-\mu$. 

It is clear that $\mu$ is non-negative on $A$ and there are sequeces $(A_i)_{i=1}^\infty$ and $(B_i)_{i=1}^\infty$ of mutually disjoint open intervals in $A$ satisfying $\mu(A_i)=\mu(B_i)={\mu(A)}/{4^i}>0$ for every $i\in \mathbb{N}$ and $A_i\cap B_j=\phi$ for every $i,j\in \mathbb{N}$.
Note that $\mu\left( A \setminus \left(\cup_{i\in \mathbb{N}} (A_i\cup B_i)\right)\right)>0$. Define a regular Borel measure $\nu$ as  
\begin{align*}
d\nu = \sum_{i=1}^\infty 2^i (\chi_{A_i} - &\chi_{B_i}) d\mu - \chi_{A \setminus \left(\cup_{i\in \mathbb{N}} (A_i\cup B_i)\right)} d\mu \\
&+ \frac{\mu\left( A \setminus \left(\cup_{i\in \mathbb{N}} (A_i\cup B_i)\right)\right)}{\delta-\gamma} \left(\chi_{\left(\gamma,\frac{\gamma+\delta}{2}\right)}d\sigma-\chi_{\left(\frac{\gamma+\delta}{2},\delta\right)} d\sigma \right), 
\end{align*}
where $\chi_{\{\cdot\}}$ is the characteristic function and $\sigma$ is the Lebesgue measure.

We first claim that $\mu$ is strongly orthogonal to $\nu$ in the sense of Birkhoff-James. Indeed, for $0<\lambda<1$, take the smallest $n_0\in\mathbb{N}$ so that $\lambda2^{n_0}> 1$. Then, we have
\begin{align*}
\|\mu+ \lambda \nu\|
&= |\mu+ \lambda \nu|([0,1] \setminus (A \cup C))+|\mu+ \lambda \nu|(\cup_{i\in \mathbb{N}} \left(A_i\cup B_i)\right) \\
&\qquad +|\mu+ \lambda \nu|\left(A \setminus (\cup_{i\in \mathbb{N}} (A_i\cup B_i)\right)+|\mu+ \lambda \nu|(C) \\
&=|\mu|([0,1] \setminus A) + (\mu+\lambda\nu)(\cup_{i< n_0} \left(A_i\cup B_i)\right)\\
&\qquad +(\mu+\lambda\nu)\left(\cup_{i\geq n_0} A_i\right)-(\mu+\lambda\nu)\left(\cup_{i\geq n_0} B_i\right)\\
&\qquad\qquad +(1-\lambda)\mu\left( A \setminus \left(\cup_{i\in \mathbb{N}} (A_i\cup B_i)\right)\right)+\lambda\mu\left( A \setminus \left(\cup_{i\in \mathbb{N}} (A_i\cup B_i)\right)\right)\\
&= |\mu|([0,1] \setminus \left(\cup_{i\geq n_0} (A_i\cup B_i)\right))+\sum_{i\geq n_0}\lambda 2^i (\mu(A_i)+\mu(B_i)) >\|\mu\|.
\end{align*}
For $\lambda<0$, we have
\begin{align*}
\|\mu+ \lambda \nu\|
&= |\mu+ \lambda \nu|([0,1] \setminus (A \cup C))+|\mu+ \lambda \nu|(\cup_{i\in \mathbb{N}} \left(A_i\cup B_i)\right) \\
&\qquad +|\mu+ \lambda \nu|\left(A \setminus (\cup_{i\in \mathbb{N}} (A_i\cup B_i)\right)+|\mu+ \lambda \nu|(C) \\
&\geq |\mu|([0,1] \setminus A ) + (\mu+\lambda\nu)(\cup_{i\in \mathbb{N}} \left(A_i\cup B_i)\right)\\
&\qquad +(1+|\lambda|)\mu\left( A \setminus \left(\cup_{i\in \mathbb{N}} (A_i\cup B_i)\right)\right)+|\lambda|\mu\left( A \setminus \left(\cup_{i\in \mathbb{N}} (A_i\cup B_i)\right)\right)\\
&=|\mu|([0,1])+2|\lambda|\mu\left( A \setminus \left(\cup_{i\in \mathbb{N}} (A_i\cup B_i)\right)\right) >\|\mu\|.
\end{align*}

Secondly, we show that $\mu(f)$ is not orthogonal to $\nu(f)$ in the sense of Birkhoff-James for an arbitrary $f\in M_\mu$ by showing that $\nu(f)\neq 0$. Without loss of generality, we assume that $\mu(f)=\|\mu\|$. It is clear that $\int_{A_i}fd\mu=\mu(A_i)=\mu(B_i)=\int_{B_i}fd\mu$ for every $i \in \N$ and $\int_{A \setminus \left(\cup_{i\in \mathbb{N}} (A_i\cup B_i)\right)} f d\mu=\mu\left( A \setminus \left(\cup_{i\in \mathbb{N}} (A_i\cup B_i)\right)\right)$. Hence, we have that
\begin{align*}
\nu(f)
&=\int_A f d \nu +\int _C f d\nu \\
&= \mu\left( A \setminus \left(\cup_{i\in \mathbb{N}} (A_i\cup B_i)\right)\right)\left(-1+ \frac{1}{\delta-\gamma}\left(\int_{\left(\gamma,\frac{\gamma+\delta}{2}\right)} f d\sigma - \int_{\left(\frac{\gamma+\delta}{2},\delta\right)} f d\sigma\right)\right)\\
&<0.
\end{align*}

The last inequality follows from that fact that $\int_{\left(\gamma,\frac{\gamma+\delta}{2}\right)} f d\sigma - \int_{\left(\frac{\gamma+\delta}{2},\delta\right)} f d\sigma=\delta-\gamma$ implies $f$ is $1$ on $\left(\gamma,\frac{\gamma+\delta}{2}\right)$ and $-1$ on $\left(\frac{\gamma+\delta}{2},\delta\right)$ almost everywhere and it is not possible for continuous functions.
\end{proof}

\begin{prop}\label{example2} A norm attaining regular Borel measure $\mu\in \mathcal{M}[0,1]$ satisfying the following two properties does not have the adjusted $\BS$ property. 
\begin{enumerate}
\setlength\itemsep{0.3em}
\item[\textup{(i)}] there exists an infinite subset $A \subset [0,1]$ such that $\mu(\{t\})\neq 0$ for each $t\in A$.
\item[\textup{(ii)}] there exists an interval $C=(\gamma,\delta) \subset [0,1]\setminus A$ such that $|\mu|(C)=0.$
\end{enumerate}
\end{prop}
\begin{proof}
We only give a proof for the case that there are infinitely many $t\in A$ such that $\mu(t)>0$ since the other case can be proved by taking $-\mu$. We assume that $\mu$ is positive on $A$ by considering the infinite subset for the convenience.

To construct a desired regular Borel measure $\nu$ as in Proposition \ref{example1}, we first pick an element $z\in A$. Since $\|\mu\|<\infty$, there are disjoint subsets $A_1=\{x_i : i\in \mathbb{N}\}$ and $A_2=\{y_i : i\in \mathbb{N}\}$ of $A\setminus\{z\}$ satisfying $\max\{\mu(\{x_i\}),\mu(\{y_i\})\}\leq 1/4^i$ for every $i\in \mathbb{N}$. Define a regular Borel measure $\nu$ as
$$d\nu=\sum_{i=1}^\infty \frac{1}{2^i}d\delta_{x_i}-\sum_{i=1}^\infty \frac{1}{2^i}d\delta_{y_i}-\mu(\{z\})d\delta_z+\frac{\mu(\{z\})}{\delta-\gamma}\left(\chi_{\left(\gamma,\frac{\gamma+\delta}{2}\right)}d\sigma-\chi_{\left(\frac{\gamma+\delta}{2},\delta\right)} d\sigma \right)$$
where $\delta_{\{\cdot\}}$ is the Dirac measure, $\chi_{\{\cdot\}}$ is the characteristic function, and $\sigma$ is the Lebesgue measure.

We claim that $\mu$ is strongly orthogonal to $\nu$ in the sense of Birkhoff-James. Indeed, for $0<\lambda<1$, take a number $n_0\in\mathbb{N}$ so that $\lambda2^{n_0}> 1$. It is obvious that $\lambda \nu(\{x_i\})=-\lambda\nu(\{y_i\})>\max\{\mu(\{x_i\}),\mu(\{y_i\})\}$ for every $i\geq n_0$. Then, we have
\begin{align*}
\|\mu+ \lambda \nu\|
&= |\mu+ \lambda \nu|([0,1] \setminus (A_1\cup A_2 \cup \{z\} \cup C))+|\mu+ \lambda \nu|(A_1\cup A_2) \\
&\qquad +|\mu+ \lambda \nu|\left(\{z\}\right)+|\mu+ \lambda \nu|(C) \\
&\geq |\mu|([0,1] \setminus (A_1\cup A_2 \cup \{z\})) + (\mu+\lambda\nu) \left(\{x_i, y_i : i< n_0\}\right)\\
&\qquad +(\mu+\lambda\nu)\left(\{x_i : i\geq n_0\}\right)-(\mu+\lambda\nu)\left(\{y_i : i\geq n_0\}\right)\\
&\qquad\qquad +(1-\lambda)\mu\left( \{z\}\right)+\lambda\mu\left( \{z\}\right)\\
&= |\mu|([0,1] \setminus \{y_i : i\geq n_0\})+\sum_{i\geq n_0}\left(2\lambda|\nu(\{y_i\})|-\mu(\{y_i\})\right) >\|\mu\|.
\end{align*}
For $\lambda<0$, we have
\begin{align*}
\|\mu+ \lambda \nu\|
&= |\mu+ \lambda \nu|([0,1] \setminus (A_1\cup A_2 \cup \{z\} \cup C))+|\mu+ \lambda \nu|(A_1\cup A_2) \\
&\qquad +|\mu+ \lambda \nu|\left(\{z\}\right)+|\mu+ \lambda \nu|(C) \\
&\geq |\mu|([0,1] \setminus (A_1\cup A_2 \cup \{z\})) + (\mu+\lambda\nu)(A_1\cup A_2)\\
&\qquad +(1+|\lambda|)\left(\{z\}\right)+|\lambda|\mu\left(\{z\}\right)\\
&=|\mu|([0,1])|+2|\lambda||\mu|(\{z\}) >\|\mu\|.
\end{align*}

We now show that $\mu(f)$ is not orthogonal to $\nu(f)$ in the sense of Birkhoff-James for an arbitrary $f\in M_\mu$ by showing that $\nu(f)\neq 0$ as we did in Proposition \ref{example1}. Arguing similarly, we may assume such $f$ has the value $1$ on $A$. Hence, we have that
\begin{align*}
\nu(f)
&=\int_{A_1\cup A_2} f d\nu +\int_{\{z\}} f d\nu+\int _C f d\nu \\
&= \mu\left( \{z\}\right)\left(-1+ \frac{1}{\delta-\gamma}\left(\int_{\left(\gamma,\frac{\gamma+\delta}{2}\right)} f d\sigma - \int_{\left(\frac{\gamma+\delta}{2},\delta\right)} f d\sigma\right)\right) <0.
\end{align*}
\end{proof}

From the Bishop-Phelps theorem \cite{BP}, it is true that the set of norm attaining functionals is norm-dense for arbitrary Banach spaces. Moreover, for a norm attaining functional on $C[0,1]$, it is easy to find another functional which is arbitrarily close to the previous one and satisfies the conditions in Proposition \ref{example1} or \ref{example2}. Hence, the set of norm attaining functionals without the adjusted $\BS$ property is norm-dense. It is worth mentioning that in $\Lin(c_0,\R)$ there is no norm attaining operators without the adjusted $\BS$ property since $\BSa(c_0,\R)=\NA(c_0,\R)$ as in Theorem \ref{thm:ell1predual}.

\begin{cor}
The set of norm attaining operators without the adjusted $\BS$ property is norm-dense in $\Lin(C[0,1],\R)$.
\end{cor}

On the other hand, using Proposition \ref{example1} and \ref{example2}, we shall see that the denseness of $\BSa(C[0,1],\R)$ does not hold. In the proof, we make use of the following decomposition of a regular Borel measure on the real line $\mathbb{R}$. 

\begin{lemma}[\mbox{\cite[Theorem 19.57, Ch V]{HS}}]\label{lem:HS}
Let $\mu$ be any regular Borel measure on $\mathbb{R}$. Then $\mu$ can be expressed in exactly one way in the form 
\[
\mu = \mu_c + \mu_d, 
\]
where $\mu_c$ is a continuous regular Borel measure and $\mu_d$ is a purely discontinuous measure. 
\end{lemma} 
Here, a continuous measure is a measure whose value is $0$ at every singleton, whereas a purely discontinuous measure is a (possibly infinite) countable linear combination of Dirac measures. Note that the cumulative distribution of a continuous measure is continuous (see \cite[Remark 19.58, Ch V]{HS}).

\begin{theorem}
\label{nondensec}The set $\BSa(C[0,1],\R)$ is not norm-dense in $\Lin(C[0,1],\R)$.
\end{theorem}

\begin{proof}
Define a regular Borel measure $\mu$ on $[0,1]$ as
$$
d\mu=\chi_{\left(0,\frac{1}{2}\right)}d\sigma-\chi_{\left(\frac{1}{2},1\right)} d\sigma,
$$
where $\chi_{\{\cdot\}}$ is the characteristic function and $\sigma$ is the Lebesgue measure.
It is clear that $\mu$ vanishes at $\pm\chi_{[0,1]}$. Hence, every norm attaining $\nu\in\mathcal{M}[0,1]$ with $\|\mu-\nu\|<1/2$ does not attain its norm at these characteristic functions. Observe that $f^{-1}((-1,1))$ is a non-empty open set whose total variation $|\nu|\bigl(f^{-1}((-1,1))\bigr)$ with respect to $\nu$ is $0$ for any norm attaining point $f\in S_{C[0,1]}$ of $\nu$ (see the proof of \cite[Proposition 3.7]{KL} for details). Moreover, this implies that $|\nu| (f^{-1}(\{1\})) >0$ and $|\nu|(f^{-1}(\{-1\}))>0$, since otherwise $\nu$ attains its norm at $\pm\chi_{[0,1]}$. 

Assume that there exists $\nu \in \BSa (C[0,1], \R)$ such that $\| \mu -\nu \| < 1/2$. Fix $f\in S_{C[0,1]}$ such that $\nu(f)=\|\nu\|$ and let $C \subset f^{-1}((-1/2,1/2)) \subset [0,1]$ be an open interval which satisfies $|\nu|(C)=0$.
According to Lemma \ref{lem:HS}, we can write $\nu^+=\nu_c^++\nu_d^+$ and $\nu^-=\nu_c^-+\nu_d^-$ where $\nu^+$ and $\nu^-$ are restrictions of $\nu$ on $f^{-1}((1/2,\infty))$ and $f^{-1}((-\infty,-1/2))$ respectively and $\nu_c^{\{\cdot\}}$ is the continuous measure and $\nu_d^{\{\cdot\}}$ is the purely discontinuous one of each. We note that $\nu^{+}$ and $-\nu^{-}$ are positive measures.

If one of $\nu_d^{\{\cdot\}}$ is not zero on infinitely many points, we see that $\nu$ does not have the adjusted $\BS$ property from Proposition \ref{example2}, which is a contradiction. Therefore, there are only finitely many discontinuities of $\nu$.

We now suppose that $\nu^{+}_c$ is not $0$, then there is an open interval $A = (\eta_1,\eta_2) \subset f^{-1}((1/2,\infty))$ so that $\nu^{+}_c\bigl((\eta_1,t)\bigr) \geq 0$ for every $t \in A$ and $\nu^{+}_c\bigl((\eta_1,\eta_2)\bigr)>0$. Since there are only finitely many discontinuities of $\nu$, we may assume that $\nu^+$ and $\nu^+_c$ are the same on $A$. This shows that the measure $\nu$ satisfies the conditions of Proposition \ref{example1}, which is a contradiction. Similarly, we deduce the same result for the case that $\nu^{-}_c$ is not $0$.

Consequently, we see that the only candidate of measure $\nu$, with $\| \mu - \nu\| < 1/2$, having the adjusted $\BS$ property is a finite linear combination of Dirac measures. However, we then have $\|\mu-\nu\|=\|\mu\|+\|\nu\|\geq \|\mu\|= 1$ for such $\nu$.
\end{proof}

Even though the set $\BSa(C[0,1],\R)$ is not norm-dense in $\Lin(C[0,1],\R)$, we see that there are a lot of operators with the adjusted $\BS$ property. Indeed, we show that $\BSa(C[0,1],\R)$ is weak-$*$-dense.

\begin{prop}
\label{example3} If a measure $\mu\in \mathcal{M}[0,1]$ is a finite linear combination of Dirac measures, then it has the adjusted $\BS$ property.
\end{prop}

\begin{proof} The proof is similar to that of Theorem \ref{thm:ell1predual}, but we give the details for the completeness. Let us write $\mu = \sum_{k=1}^n a_k \delta_{x_k}$ with non-zero $a_k$ and elements $x_k\in [0,1]$ where $\delta_{\{\cdot\}}$ is the Dirac measure. Suppose that $\mu \perp_S \nu$ for some $\nu\in \mathcal{M}[0,1]$. We claim first that
$$
\left| \sum_{k=1}^n \sgn (a_k) \nu(\{x_k\}) \right| < |\nu|([0,1]\setminus \{x_k : 1\leq k\leq n\}).
$$
If it is not true, for $\lam_0\in \mathbb{R}$ such that $0<|\lam_0| < \min_{1 \leq k \leq n} \frac{|a_k|}{\max\{1,|\nu|(\{x_k\})\}}$ and $\sgn \lam_0 = - \sgn(\sum_{k=1}^n \sgn (a_k) \nu(\{x_k\}))$, we have that
\begin{align*}
\|\mu + \lam_0 \nu\| &= \sum_{k=1}^n |a_k + \lam_0 \nu(\{x_k\})| + |\lam_0| |\nu|([0,1]\setminus \{x_k : 1\leq k\leq n\}) \\
&= \sum_{k=1}^n \left| |a_k| + \lam_0 \sgn(a_k) \nu(\{x_k\}) \right| + |\lam_0| |\nu|([0,1]\setminus \{x_k : 1\leq k\leq n\}) \\
&= \sum_{k=1}^n |a_k| - |\lam_0| \left| \sum_{k=1}^n \sgn(a_k) \nu(\{x_k\}) \right| + |\lam_0| |\nu|([0,1]\setminus \{x_k : 1\leq k\leq n\}) \\
&\leq \sum_{k=1}^n |a_k| = \|\mu\|,
\end{align*}
which is a contradiction. 

To find $f\in S_{C[0,1]}$ such that $\mu(f)=\|\mu\|$ and $\nu(f)=0$, we take $\eps>0$ such that
$$
\left| \sum_{k=1}^n \sgn (a_k) \nu(\{x_k\}) \right| +4\eps < |\nu|([0,1]\setminus \{x_k : 1\leq k\leq n\}).
$$
From the regularity of $\nu$ there are a compact $K\subset [0,1]\setminus \{x_k : 1\leq k\leq n\}$ and mutually disjoint open sets $U_k\subset [0,1]\setminus K$ such that $x_k \in U_k$, $\overline{U_k}\subset [0,1]\setminus K$,
$$\sum_{k=1}^n |\nu|(U_k)< \sum_{k=1}^n \left|\nu(\{x_k\}) \right| +\eps \text{~and~}|\nu|(K)> |\nu|([0,1]\setminus \{x_k : 1\leq k\leq n\})-\eps.$$

Using Urysohn's Lemma, we construct a continuous non-negative function $h$ in $S_{C[0,1]}$ such that $h \equiv 1$ on $K$ and $h \equiv 0$ on $\cup_{k=1}^n \overline{U_k}$. Similarly, we take non-negative continuous functions $h_k$ in $S_{C[0,1]}$ such that $h_k \equiv 1$ at $x_k$ and $h_k \equiv 0$ on $[0,1]\setminus U_k$.
Now, consider an element $g\in C[0,1]$ such that $\int_K g d \nu>|\nu|(K)-\eps$, and define $f_\alpha \in M_\mu$ by $f_\alpha =\sum_{k=1}^n \sgn(a_k)h_k + \alpha h g$ for $|\alpha|\leq 1$. It is clear that $\mu(f_\alpha) = \|\mu\|$ for any $|\alpha| \leq 1$. On the other hand, we have that 
\begin{align*}
\nu(f_\alpha)
&=\sum_{k=1}^n \int_{\{x_k\}}f_\alpha d\nu +\sum_{k=1}^n \int_{U_k\setminus \{x_k\}}f_\alpha d\nu +\int_{[0,1]\setminus (\cup_{k=1}^n \overline{U_k})}f_\alpha d\nu\\
&= \sum_{k=1}^n \sgn (a_k) \nu(\{x_k\})+\sum_{k=1}^n \int_{U_k\setminus \{x_k\}} \sgn (a_k) h_k d\nu +\alpha\int_{[0,1]\setminus (\cup_{k=1}^n \overline{U_k})} h g d\nu
\end{align*}
From the inequalities 
$$\left|\sum_{k=1}^n \sgn (a_k) \nu(\{x_k\})+\sum_{k=1}^n \int_{U_k\setminus \{x_k\}} \sgn(a_k) h_k d\nu\right| <\left| \sum_{k=1}^n \sgn (a_k) \nu(\{x_k\}) \right| +\eps$$
and
\begin{align*}
\left|\int_{[0,1]\setminus (\cup_{k=1}^n \overline{U_k})} h g d\nu\right|
&\geq \left|\int_{K} hg d\nu\right|-|\nu|\left([0,1]\setminus \left(K\cup\left(\cup_{k=1}^n \overline{U_k}\right)\right)\right)\\
&>|\nu|(K)-2\eps>|\nu|([0,1]\setminus \{x_k : 1\leq k\leq n\})-3\eps,
\end{align*}
we see that there exists $|\alpha|\leq 1$ so that $\nu (f_\alpha)=0$. This shows that $\mu$ has the adjusted $\BS$ property.
\end{proof}

 It is well known that the set of extreme points of $B_{\mathcal{M}[0,1]}$ is the set of Dirac measures on $[0,1]$ (see \cite[Lemma 3.42]{FHH}). From Proposition \ref{example3}, we see that all the linear combinations of extreme points are contained in the set $\BSa(C[0,1],\R)$, so we have shown the following consequence due to the Krein-Milman theorem.

\begin{cor}\label{cor:C[0,1]}The set $\BSa(C[0,1],\R)$ is weak-$*$-dense in $\Lin(C[0,1],\R)$.
\end{cor}

\begin{remark}
Note that the argument used in Proposition \ref{example3} (and in Corollary \ref{cor:C[0,1]}) can be applied to a general $C(K)$-space. Thus, we observe that $\BSa (\ell_\infty, \mathbb{R})$ is weak-$*$-dense in $\Lin (\ell_\infty, \mathbb{R})$ since $\ell_\infty = C(\beta \N)$. 
\end{remark} 

Since it is known that $\BS(C[0,1],\R) \subseteq \left\{ T \in \NA(C[0,1],\R) : M_T =\left\{ \pm \chi_{[0,1]}\right\}\right\} \cup \{0\}$ (see \cite{KL}), we have $\BS(C[0,1],\R) \subsetneqq \BSa(C[0,1],\R)$, and moreover, the set $\BS(C[0,1],\R)$ is not weak-$*$-dense in $\Lin(C[0,1],\R)$. 
\begin{prop}The set $\BS(C[0,1],\R)$ is not weak-$*$-dense in $\Lin(C[0,1],\R)$.
\end{prop}
\begin{proof}
Let $\mu$ be a measure whose norm is $1$ and vanishes at $\pm\chi_{[0,1]}$ like the one in the beginning of the proof of Theorem \ref{nondensec}. Take $f \in B_{C[0,1]}$ such that $\mu(f)>1/2$.
Observe from \cite[Proposition 3.7]{KL} which is mentioned in the beginning of the present section that $\BS(C[0,1],\R)\cap \left\{\nu\in \mathcal{M}[0,1] : |\nu\left(\chi_{[0,1]}\right)|<1/2\right\}$ is contained in a ball $(1/2)B_{\mathcal{M}[0,1]}$. Hence, the set 
\[
\left\{\nu\in \mathcal{M}[0,1]~:~ |\nu\left(\chi_{[0,1]}\right)|<\frac{1}{2},~\nu(f)>\frac{1}{2}\right\}
\]
is a weak-$*$-open set containing $\mu$ which does not intersect with $\BS(C[0,1],\R)$.
\end{proof}

This notable difference between the set of operators with the $\BS$ property and the set of operators with the adjusted $\BS$ property may lead to the following general question.

\begin{question}
Is $\BSa(X,\R)$ weak-$*$-dense in $\Lin(X,\R)$ for every Banach space $X$?
\end{question}

\subsection*{Statements \& Declarations}

\subsection*{Funding}
The first author was supported by Basic Science Research Program through the National Research Foundation of
Korea(NRF) funded by the Ministry of Education, Science and Technology [NRF-2020R1A2C1A01010377]. The second author was supported by NRF (NRF-2019R1A2C1003857), by POSTECH Basic Science Research Institute Grant (NRF-2021R1A6A1A10042944) and by a KIAS Individual Grant (MG086601) at Korea Institute for Advanced Study. The third author was supported by the National Research Foundation of Korea(NRF) grant funded by the Korea government(MSIT) [NRF-2020R1C1C1A01012267].

\subsection*{Competing Interests} The authors have no relevant financial or non-financial interests to disclose.

\subsection*{Author Contributions} All authors contributed to the whole part of works together such as the study conception, design and writing. All authors read and approved the final manuscript.

\end{document}